\documentclass{article}
\usepackage[utf8]{inputenc}
\usepackage{amsmath}
\usepackage{amsfonts}
\usepackage{amssymb}
\usepackage{amsthm}
\usepackage[all]{xy}
\usepackage{enumerate}
\usepackage{cite}
\usepackage[affil-it]{authblk}
\usepackage{hyperref}
\usepackage{bbm}
\hypersetup{hypertex=true,
            linkcolor=blue,
            anchorcolor=blue,
            citecolor=blue}

\theoremstyle{plain}
  \newtheorem{thm}{Theorem}[section]
  \newtheorem{lem}[thm]{Lemma}
  \newtheorem{prop}[thm]{Proposition}
   
\theoremstyle{definition}
  \newtheorem{defn}[thm]{Definition}
  \newtheorem{rmk}[thm]{Remark}
  \newtheorem{ex}[thm]{Example}

\theoremstyle{plain}

  \theoremstyle{plain}
  \newtheorem{que}{Question}

\DeclareMathOperator{\im}{im\,}
 
\numberwithin{equation}{section}

\allowdisplaybreaks

\title{Cohomology of minimal Sullivan algebras of non-finite type and their realizations}
\author{Jiawei Zhou}
\date{September 25, 2024}

\begin{document}

\maketitle

\begin{abstract}
We prove that the morphisms from a minimal Sullivan algebra $\Lambda V$ to $A_{PL}(|\Lambda V|)$, the algebra of polynomial differential forms on its realization, can be quasi-isomorphic if and only if the cohomology $H(\Lambda V)$ is of finite type. Importantly, $\Lambda V$ itself need not be of finite type. For example, it can be the minimal Sullivan model of the wedge sum of a circle and a sphere. This provides a negative answer to a question posed by F\'elix, Halperin, and Thomas. Furthermore, we study the spaces whose homotopy groups are reflected by their minimal Sullivan models as a generalization of Sullivan spaces, and explore which properties of Sullivan spaces can be broadened.
\end{abstract}

\section{Introduction}

The algebraic models formulated by Quillen \cite{Quillen} and Sullivan \cite{Sullivan} serve as cornerstones in rational homotopy theory. Quillen's model is a differential graded Lie algebra constructed from the rational homotopy groups of a simply connected space, while Sullivan's model is a commutative differential graded algebra (CDGA) built upon the rational cohomology ring. Every path-connected space possesses a unique CDGA (up to isomorphism) of special type, known as its minimal Sullivan model.

The minimal Sullivan model of a space $X$ is a free graded commutative algebra $\Lambda V$ equipped with a differential that satisfies certain conditions. There need to be a quasi-isomorphism from $\Lambda V$ to $A_{PL}(X)$, the CDGA of polynomial differential forms on X. This quasi-isomorphism is sometimes also referred to as the minimal Sullivan model of $X$.

Moreover, when $X$ is simply connected and $H^*(X,\mathbb{Q})$ is of finite type, the rational homotopy group $\pi_n(X)\otimes\mathbb{Q}$ can be readily computed from $\Lambda V$. This homotopy group is isomorphic to the dual space of $V^n$, where $V^n$ denotes the subspace of $V$ in degree $n$. These dual spaces are also defined as the homotopy group of $\Lambda V$ and written as $\pi_n(\Lambda V)$.

Conversely, any simply connected (with trivial first cohomology) CDGA is quasi-isomorphic to a unique special CDGA $\Lambda V$, known as the minimal Sullivan algebra. When $\Lambda V$ is of finite type, one can construct a topological space $|\Lambda V|$ such that $\pi_n(|\Lambda V|)=\pi_n(|\Lambda V|)\otimes\mathbb{Q}$ is isomorphic to the homotopy group of $\Lambda V$, and $H^n(|\Lambda V|;\mathbb{Z}) = H^n(|\Lambda V|;\mathbb{Q}) = H^n(\Lambda V;\mathbb{Q})$. The space $|\Lambda V|$ is referred to as the realization of $\Lambda V$. If $\Lambda V$ is a minimal Sullivan model of some space $X$, then $|\Lambda V|$ and $X$ have the same rational homotopy type.

Therefore, the minimal Sullivan model and the realization establish a bijective correspondence between the rational homotopy types of simply connected spaces with finite-type cohomology and the isomorphism classes of simply connected minimal Sullivan algebras of finite type.

For the general case, such an identification no longer holds. While we can still construct a minimal Sullivan model based on the cohomology of a space, their rational homotopy groups may differ. Similarly, although we can construct the realization of a minimal Sullivan algebra, their rational cohomologies may not necessarily be isomorphic. These observations lead to the following questions.
\begin{itemize}
\item Given a minimal Sullivan algebra $\Lambda V$, under what conditions is the map $\Lambda V\to A_{PL}(|\Lambda V|)$ a quasi-isomorphism?
\item Given a topological space $X$ with a minimal Sullivan model $\Lambda V$, when is $\pi_n(X)\otimes\mathbb{Q} \to \pi_n(\Lambda V)$ isomorphic for all $n\geq 2$? Furthermore, under what conditions can we construct certain isomorphisms based on $\pi_1(X)$ and the homotopy Lie algebra of $\Lambda V$?
\end{itemize}

When $H^1(\Lambda V;\mathbb{Q})$ and $V^n $ for $n\geq 2$ are all finite-dimensional, these questions have been answered affirmatively. The corresponding path-connected space $X$ satisfies $\dim H^1(X;\mathbb{Q})<\infty$ and $H^*(\widetilde{X};\mathbb{Q})$ is of finite type, where $\widetilde{X}$ is the universal cover of $X$. Such space a $X$ is called a Sullivan space if its minimal Sullivan model $\Lambda V \stackrel{\simeq}{\longrightarrow} A_{PL}(X)$ induces isomorphisms $\pi_n(X)\otimes\mathbb{Q} \stackrel{\cong}{\longrightarrow} \pi_n(\Lambda V)$ for all $n\geq 2$. Several equivalent conditions for being a Sullivan space exist \cite[Theorem 7.1 and 7.2]{FHT2}. On the other hand, if $\Lambda V$ satisfies the above hypotheses, then the induced map $m_{|\Lambda V|}: \Lambda V\to A_{PL}(|\Lambda V|)$ is a quasi-isomorphism if and only if $|\Lambda V|$ is a Sullivan space \cite[Theorem 7.8]{FHT2}.

In the case where $\Lambda V$ is simply connected, $m_{|\Lambda V|}:\Lambda V\to A_{PL}(|\Lambda V|)$ is a quasi-isomorphism if and only if $V$ is of finite type \cite[Theorem 1.5]{FHT2}. Consequently, if the minimal Sullivan model $\Lambda V$ of a simply connected space $X$ induces isomorphisms on rational homotopy groups, then $H^*(X;\mathbb{Q})$ must be of finite type. This is because, under these conditions, $X$ and $|\Lambda V|$ have the same rational homotopy type.

It remains to consider the case where the space is neither simply connected nor of finite type. In \cite{FHT2}, F\'elix, Halperin, and Thomas posed the question of whether there exist minimal Sullivan algebras that are quasi-isomorphic to their realizations, even if they are not simply connected or of finite type.

\begin{que}\label{FHT question}
Suppose that $\Lambda V$ is a minimal Sullivan model of a path-connected space $X$, and the induced map $X\to |\Lambda V|$ is an isomorphism on rational cohomologies. Is $|\Lambda V|$ necessarily a Sullivan space? If so, must $X$ also be a Sullivan space?
\end{que}

When $H^1(\Lambda V;\mathbb{Q})$ and $V^n$ for $n\geq 2$ are finite dimensional, Theorem 7.8 of \cite{FHT2} implies that $|\Lambda V|$ must be a Sullivan space. Thus, the first part of Question \ref{FHT question} essentially examines whether the map $m_{|\Lambda V|}:\Lambda V \to A_{PL}(|\Lambda V|)$ can be a quasi-isomorphism when $\dim V^n=\infty$ for some $n\geq 2$. We will see that this is possible, provided that $H^*(\Lambda V)$ is of finite type.

\begin{thm}\label{intro cohomology of minimal Sullivan and realization}
Let $\Lambda V$ be a minimal Sullivan algebra. The following statements are equivalent.
\begin{enumerate}[(i)]
\item There exists a quasi-isomorphism $\Lambda V \stackrel{\simeq}{\longrightarrow} A_{PL}(|\Lambda V|)$.
\item $m_{|\Lambda V|}:\Lambda V \to A_{PL}(|\Lambda V|)$ is a quasi-isomorphism.
\item $H^*(\Lambda V)$ has finite type.
\end{enumerate}
\end{thm}
Thus, the minimal Sullivan model of $S^1\vee S^2$ provides a counterexample to the first part of Question \ref{FHT question}, as explained in Example \ref{eg S1 vee S2}.

Since realization and taking minimal Sullivan model also send corresponding morphisms to each other, we can generalize Theorem \ref{intro cohomology of minimal Sullivan and realization} using the language of category theory.

\begin{thm}
Let $\mathcal{C},\mathcal{D}$ be categories, and let $\mathcal{F},\mathcal{G}$ be contravariant functors defined as follows.

$\mathrm{Ob}(\mathcal{C})$: isomorphism classes of minimal Sullivan algebras with finite type cohomology.

$\mathrm{Mor}(\mathcal{C})$: homotopy classes of CDGA morphisms.

$\mathrm{Ob}(\mathcal{D})$: path connected spaced with finite type cohomology.

$\mathrm{Mor}(\mathcal{D})$: homotopy classes of continuous maps.

$\mathcal{F}:\mathcal{C}\to \mathcal{D}$, realization.

$\mathcal{G}:\mathcal{D}\to \mathcal{C}$, take minimal Sullivan models.

Then $\mathcal{F}:\mathcal{C}\to\mathcal{D}$ is a full and faithful functor, with a left inverse given by $\mathcal{G}$.
\end{thm}

We also investigate the relationship between the homotopy groups of a topological space and its minimal Sullivan model. For a minimal Sullivan model $\Lambda V \stackrel{\simeq}{\longrightarrow} A_{PL}(X)$ to induce isomorphisms $\pi_n(X)\otimes\mathbb{Q} \stackrel{\cong}{\longrightarrow} \pi_n(\Lambda V)$ for $n\geq 2$, none of the critical properties of Sullivan spaces from Theorem \ref{thm 7.2 FHT} necessarily hold. In fact, it is possible that $\dim H^1(X; \mathbb{Q})=\infty$. However, this situation can be excluded if additional constraints are imposed on $\pi_1(X)$, such as requiring that its Malcev $\mathbb{Q}$-completion \cite{malcev} is the fundamental group of the minimal model. When $\dim H^1(X; \mathbb{Q})<\infty$, the isomorphism $\pi_n(X)\otimes\mathbb{Q} \cong \pi_n(\Lambda V)$ holds only if $\dim H^*(X; \mathbb{Q})$ is of finite type.

\begin{prop}\label{intro finite type of H(X)}
Let $f:\Lambda V \stackrel{\simeq}{\longrightarrow} A_{PL}(X)$ be a minimal Sullivan model of a topological space $X$, and suppose $\dim H^1(X; \mathbb{Q}) = \dim H^1(\Lambda V) < \infty$. If the induced map $\pi_n(f):\pi_n(X)\otimes\mathbb{Q} \to \pi_n(\Lambda V)$ is an isomorphism for all $n\geq 2$, then $H^*(X; \mathbb{Q})$ is of finite type.
\end{prop}

Finally, we present an example of a non-Sullivan space whose minimal Sullivan model can be realized as a Sullivan space. This provides a negative answer to the second part of Question \ref{FHT question}.

This paper is organized as follows. In Section 2, we review the realization of minimal Sullivan algebras, the holonomy actions of fibrations and their algebraic representations, as well as Sullivan spaces. These topics lay the foundation for proving Theorem \ref{intro cohomology of minimal Sullivan and realization} in Section 3. In Section 4, we explore the relationship between the homotopy groups of a topological space and its minimal Sullivan model. Throughout these sections, we assume the ground ring is $\mathbb{Q}$, and hence, for brevity, we denote $H^*(-;\mathbb{Q})$ simply as $H^*(-)$.

The author would like to thank Yves F\'elix, Jianfeng Lin, Jie Wu and Juxin Yang for helpful discussions. This research is supported by the National Key Research and Development Program of China No. 2020YFA0713000.

\section{Preliminary}
\subsection{Realization of minimal Sullivan algebras}
A \textbf{Sullivan algebra} is a free graded algebra $\Lambda V$ generated by a graded vector space $V$ of positive degree, equipped with a differential $d$ such that $V$ has a filtration
$$
0=V(-1) \subset V(0) \subset V(1) \subset \ldots \subset V(k) \subset \ldots
$$
and $dV(k)\subset\Lambda V(k-1)$. For simplicity, we will use $\Lambda V$ to denote the Sullivan algebra $(\Lambda V,d)$.

We call $\Lambda V$ \textbf{minimal} if $dV\subset \Lambda^{\geq 2} V$. Here $\Lambda^{\geq n} V$ denotes the subspace of $\Lambda V$ spanned by elements with wordlength at least $n$.

The \textbf{simplicial commutative cochain algebra}, $A_{PL}$, is defined as follows. $(A_{PL})_n$ is the quotient algebra of the free graded algebra generated by $t_0,\ldots,t_n,dt_0,\ldots,dt_n$, with the identification $t_0+\ldots+t_n=1$. Here $t_0,\ldots,t_n$ are all of degree 0. So $dt_0,\ldots,dt_n$ are degree-1 coboundaries and $dt_0+\ldots+dt_n=0$. The face and degeneracy morphisms of $A_{PL}$ are defined by 
$$
\partial_i t_k =
\begin{cases}
t_k & \text{if } k<i, \\
0 & \text{if } k=i, \\
t_{k-1} & \text{if } k>i,
\end{cases}
\quad\text{and}\quad
s_j t_k =
\begin{cases}
t_k & \text{if } k<j, \\
t_k+t_{k+1} & \text{if } k=j, \\
t_{k+1} & \text{if } k>j.
\end{cases}
$$

Given a Sullivan algebra $\Lambda V$, we can construct a simplical set $\langle \Lambda V \rangle$. Define $\langle \Lambda V \rangle_n$ as the morphisms of CDGAs from $\Lambda V$ to $(A_{PL})_n$. For each $\sigma\in\langle \Lambda V \rangle_n$, set $\partial_i(\sigma)=\partial_i\circ\sigma$ and $s_j(\sigma)=s_j\circ\sigma$.

By the Milnor realization, $\langle \Lambda V \rangle$ induces a CW complex $|\Lambda V|$. More generally, given any simplicial set $K$, its Milnor realization $|K|$ is defined as follows. Let $e_0,e_1,\ldots$ be the standard basis of $\mathbb{R}^{\infty}$ and let $\Delta^n$ denote the Euclidean simplex with vertices $e_0,e_1,\ldots,e_n$. Then
$$
|K| = \bigsqcup_n \, (K_n\times\Delta^n)/\sim
$$
where each $K_n$ is equipped with the discrete topology and $\sim$ is the equivalence relation generated by
$$
\partial_i\sigma\times x \sim \sigma\times\mu_i x, \quad s_j\sigma\times y \sim \sigma\times\pi_j y.
$$
Here $\mu_i:\Delta^{n-1}\to\Delta^n$ and $\pi_j:\Delta^{n+1}\to\Delta^n$ are defined by
$$
\mu_i e_k =
\begin{cases}
e_k & \text{if } k<i, \\
e_{k+1} & \text{if } k\geq i,
\end{cases}
\quad
\pi_j e_k =
\begin{cases}
e_k & \text{if } k\leq j, \\
e_{k-1} & \text{if } k>j.
\end{cases}
$$

The homotopy group $\pi_n(|\Lambda V|)$ is related to $(V^n)^{\sharp}$, where $V^n$ is the degree $n$ subspace of $V$ and $\sharp$ denotes the dual space.

\begin{thm}[Theorem 1.3 and 1.4 of \cite{FHT2}]\label{bij between homotopies}
Let $\Lambda V$ be a minimal Sullivan algebra. There exists natural bijective maps $\iota_n:\pi_n(|\Lambda V|) \to (V^n)^{\sharp}$ satisfying the following properties.
\begin{enumerate}[(i)]
\item $\iota_n$ are natural transformations.
\item For $n\geq 2$, $\iota_n$ is an isomorphism of abelian groups.
\item If $H^1(\Lambda V)$ is finite dimensional, $\text{exp}_L\circ\iota_1$ is an isomorphism of groups. Here $\text{exp}_L$ is a bijective map on $(V^1)^{\sharp}$, and equips its image with a group structure.
\end{enumerate}
\end{thm}

For $n\geq 2$, $(V^n)^{\sharp}$ is called the $n$\textbf{-th homotopy group} of $\Lambda V$ and is denoted by $\pi_n(\Lambda V)$. To define the \textbf{fundamental group} of $\Lambda V$, we need to restrict to the case where $\dim H^1(\Lambda V)<\infty$. It is given by the image of $\text{exp}_L$ and is denoted by $\pi_1(\Lambda V)$.

In particular, when $V$ is concentrated in degree 1 and $\dim H^1(\Lambda V)<\infty$, $|\Lambda V|$ is the classifying space $B\pi_1(\Lambda V)$ of $\pi_1(\Lambda V)$.

On the other hand, for any topological space $X$, the set of singular simplices $Sing \, X$ forms a simplicial set. The functors $|\bullet|$ and $Sing$ are adjoint between the categories of pointed simplicial sets and pointed topological spaces. Specifically, if $(K,u_0)$ is a pointed simplicial set and $(X,x_0)$ is a pointed topological space, then the based continuous maps $(|K|,u_0)\to(X,x_0)$ and the based simplicial maps $(K,u_0)\to(Sing \, X,x_0)$ are naturally identified.

Thus, the identity map on $|K|$ is adjoint to a simplicial map $\lambda_K: K\to Sing \, |K|$. In particular, $\lambda_K$ maps $\sigma\in K_n$ to the singular simplex $\Delta^n \to |K|$ given by $x\mapsto \sigma\times x$.

Let $S_n(K)$ be the $\mathbb{Q}$-vector space with basis $K_n$, and let $D_n(K)$ be the subspace spanned by the degenerate simplices. Then the operator $\partial=\sum_{i=0}^n (-1)^i\partial_i$ defines a chain complex $C_*(K)$, where $C_n(K)=S_n(K)/D_n(K)$. Its dual, $C^*(K)=(C_*(K))^{\sharp}=\text{Hom}(C^*(K), \mathbb{Q})$, is a cochain complex. The map $\lambda_K$ induces isomorphisms $H_*(C_*(K))\stackrel{\cong}{\longrightarrow} H_*(|K|)$ and $H^*(|K|) \stackrel{\cong}{\longrightarrow} H^*(C^*(K))$.

The space of simplicial maps from $K$ to $A_{PL}$ also has a cochain algebra structure, which is called the \textbf{cochain algebra of polynomial forms on $K$}, and is denoted by $A_{PL}(K)$. An element $\Phi\in A_{PL}^k(K)$ of degree $k$ assigns to each $\sigma\in K_n$ an element $\Phi_{\sigma} \in (A_{PL})_n^k$, satisfying the conditions
$$
\partial_i\Phi_{\sigma} = \Phi_{\partial_i\sigma}, \quad s_j\Phi_{\sigma} = \Phi_{s_j\sigma}.
$$
The differential on $A_{PL}(K)$ is defined by $(d\Phi)_{\sigma}=d(\Phi_{\sigma})$.

There exists a quasi-isomorphism $\oint_K:A_{PL}(K)\stackrel{\simeq}{\longrightarrow} C^*(K)$ given by $(\oint_K \Phi)(\sigma)=\oint_n \Phi_{\sigma}$. Here $\oint_n:A_{PL}^n\to\mathbb{Q}$ is defined in the section 10 of \cite{FHT}. The properties of $\oint_n$, which are necessary for this paper, will be described later.

So $\oint_{Sing \, |K|}$ also induces an isomorphism $H^*(A_{PL}(Sing \, |K|))\stackrel{\cong}{\longrightarrow} H^*(|K|)$. Note that $A_{PL}(Sing \, |K|)$ is commonly denoted as $A_{PL}(|K|)$ for simplicity. In summary, we obtain the following commutative diagram of natural isomorphisms.

\begin{equation}\label{identify cohomologies}
\xymatrix@C=60pt{
H^*(|K|) \ar[r]^{H(C^*(\lambda_K))}_{\cong} & H^*(C^*(K)) \\
H^*(A_{PL}(|K|)) \ar[u]^{H(\oint_{Sing \, |K|})}_{\cong} \ar[r]^{H(A_{PL}(\lambda_K))}_{\cong} & H(A_{PL}(K)) \ar[u]^{H(\oint_K)}_{\cong}
}
\end{equation}

Consider a minimal Sullivan algebra $\Lambda V$ again. The identity map on $\langle \Lambda V \rangle$ is adjoint to a map $m_{\langle \Lambda V \rangle}: \Lambda V \to A_{PL}(\langle \Lambda V \rangle)$ defined by $(m_{\langle \Lambda V \rangle}(x))_{\sigma}=\sigma(x)$ for $x\in\Lambda V$ and $\sigma\in\langle \Lambda V \rangle$. It can be shown that the quasi-isomorphism $A_{PL}(\lambda_{\langle \Lambda V \rangle}): A_{PL}(|\Lambda V|) \to A_{PL}(\langle \Lambda V \rangle)$ is surjective. Consequently, $m_{\langle \Lambda V \rangle}$ has a lift $m_{|\Lambda V|}:\Lambda V\to A_{PL}(|\Lambda V|)$ such that $m_{\langle \Lambda V \rangle} = A_{PL}(\lambda_{\langle \Lambda V \rangle}) \circ m_{|\Lambda V|}$ \cite[Proposition 1.10]{FHT2}.

Since $\Lambda V$ is minimal, all coboundaries are in $\Lambda^{\geq 2}V$. So the map $\xi_n:H^n(\Lambda V) \to (\Lambda^{\geq 1}V/\Lambda^{\geq 2}V)^n \cong V^n$, which sends cohomology classes to the classes of representatives, is well-defined.

\begin{prop}[Proposition 1.19 of \cite{FHT2}]
Let $\Lambda V$ be a minimal Sullivan algebra. For $n\geq 2$, suppose $\alpha\in\pi_n(|\Lambda V|)$ and $\beta\in H^n(\Lambda V)$. Then
$$
\langle H^n \left(m_{|\Lambda V|}\right)\beta, hur\,\alpha \rangle = \langle \xi_n(\beta), \iota_n(\alpha) \rangle.
$$

It follows that $H^n \left(m_{|\Lambda V|}\right)^{\sharp}\circ hur=\xi_n^{\sharp}\circ\iota_n$, i.e. the following diagram commutes.
\end{prop}

\begin{equation}\label{relation with hur}
\xymatrix@C=60pt{
\pi_n(|\Lambda V|) \ar[r]^{\iota_n} \ar[d]^{hur} & (V^n)^{\sharp} \ar[d]^{\xi_n^{\sharp}} \\
H_n(|\Lambda V|) \ar[r]^{H^n \left(m_{|\Lambda V|}\right)^{\sharp}} & [H^n(\Lambda V)]^{\sharp}.
}
\end{equation}

When $n=1$ and $\dim H^1(\Lambda V)<\infty$, the relation in \eqref{relation with hur} also commutes. To address the case that $\dim H^1(\Lambda V)=\infty$, we consider the map $\tau_1:(V^1)^{\sharp} \to \pi_1(|\Lambda V|)$, which is constructed in \cite{FHT2} as the inverse of $\iota_1$.

The definition of $\iota_1$ involves the simplicial set $\Delta[1]/\partial\Delta[1]$, which consists of two non-degenerate simplices $\overline{c_0}\in(\Delta[1]/\partial\Delta[1])_0$ and $\overline{c_1}\in(\Delta[1]/\partial\Delta[1])_1$. So its realization is $S^1$. In $A_{PL}(\Delta[1]/\partial\Delta[1])$, there exists an element denoted by $dt$, such that $\oint_1 \langle dt,\overline{c_1} \rangle = \oint_1 (dt)_{\overline{c_1}} = -1$. This notation comes from the identification of $A_{PL}(\Delta[1])$ with $(A_{PL})_1=\Lambda(t,dt)$ (see the Example in Section 1.4 of \cite{FHT2}).

For $f\in (V^1)^{\sharp}$, it defines a map $\varphi_f: \Lambda V\to A_{PL}(\Delta[1]/\partial\Delta[1])$ by $\varphi_f(v) = -\langle v,f \rangle dt$. Then the adjoint map $\langle \varphi_f \rangle: \Delta[1]/\partial\Delta[1] \to \langle \Lambda V \rangle$ of $\varphi_f$ realizes to a continuous map $|\varphi_f|:S^1\to|\Lambda V|$. Set $\tau_1(v)=|\varphi_f|$.

\begin{lem}\label{commutative for deg 1}
Let $\Lambda V$ be a minimal Sullivan algebra, $f\in (V^1)^{\sharp}$, and $\beta\in H^1(\Lambda V)$. Then
$$
\langle H^1 \left(m_{|\Lambda V|}\right)\beta, hur\circ\tau_1(f) \rangle = \langle \xi_1(\beta), f \rangle.
$$
In particular $H^1 \left(m_{|\Lambda V|}\right)^{\sharp}\circ hur\circ\tau_1 = \xi_1^{\sharp}$.
\end{lem}

\begin{proof}
By the construction of $\tau_1$, $hur\circ\tau_1(f)$ is the homology class of the cycle $|\varphi_f|\circ q: I\to |\Lambda V|$, where $q: I \to S^1$ is the quotient map that identifies the two end points. This homology class can be identified with the class $[\psi_f]\in H_1(C_*(\langle \Lambda V \rangle))$, where $\psi_f$ is defined as $\langle \varphi_f \rangle(\overline{c_1}) \in\langle \Lambda V \rangle$. So
$$
\oint_1 \psi_f(v) = \oint_1 \langle \varphi_f(v),\overline{c_1} \rangle = -\langle v,f \rangle \oint_1 \langle dt,\overline{c_1} \rangle = \langle v,f \rangle.
$$

On the other hand, $H^1 \left(m_{|\Lambda V|}\right)\beta$ is identified with $H^1 \left(m_{\langle\Lambda V\rangle}\right)\beta \in H^1(A_{PL}(\langle \Lambda V \rangle))$. Let $v_{\beta}\in V^1$ be the unique representative of $\beta$. Then $H^1 \left(m_{\langle\Lambda V\rangle}\right)\beta$ can be identified with the cohomology class of $\oint_{\langle\Lambda V\rangle} m_{\langle\Lambda V\rangle} (v_{\beta}) \in C^1(\langle \Lambda V \rangle)$. Therefore, we have
\begin{align*}
\langle H^1 \left(m_{|\Lambda V|}\right)\beta, hur\circ\tau_1(f) \rangle &= \left\langle \oint_{\langle\Lambda V\rangle} m_{\langle\Lambda V\rangle} (v_{\beta}), \psi_f \right\rangle \\
&= \oint_1 \langle m_{\langle\Lambda V\rangle} (v_{\beta}), \psi_f \rangle \\
&= \oint_1 \psi_f(v_{\beta}) \\
&= \langle v_{\beta},f \rangle.
\end{align*}

By definition $\xi(\beta)=v_{\beta}$, which completes the proof.
\end{proof}

\subsection{Holonomy action}

Let $F\to E\to B$ be a fibration where $F$ and $B$ are path-connected. There exists an anti-homomorphism from $\pi_1(B)$ to the homotopy classes of homotopy equivalences of $F$, known as the holonomy action. This action induces a right $\pi_1(B)$-action on $H_*(F)$ and a left $\pi_1(B)$-action on $H^*(F)$. Moreover, when $F$ is simply connected, the holonomy action also induces a right $\pi_1(B)$-action on $\pi_*(F)$.

In certain cases, fibrations can be represented by relative Sullivan algebras. Additionally, the holonomy action has an algebraic representation.

Let $\Lambda Z$ be a Sullivan algebra and $W$ a graded vector space of positive degree. Suppose that the differential on $\Lambda Z$ can be extended to $\Lambda Z\otimes\Lambda W$, and $W$ is equipped with a filtration $W(k)$ such that $dW(k) \subset \Lambda Z\otimes\Lambda W(k-1)$. Then $\Lambda Z\otimes\Lambda W$ is called a relative Sullivan algebra. More generally, $\Lambda Z$ can be replaced by any connected CDGA.

If additionally $dW^n \subset dZ\otimes\Lambda W^{\leq n}$, where $W^{\leq n}$ denotes the subspace of $W$ spanned by the elements of degree $\leq n$, then the relative Sullivan algebra is called minimal. Observe that the quotient algebra $(\Lambda Z\otimes\Lambda W)/(\Lambda^{\geq 1} Z\otimes\Lambda W)$ is a Sullivan algebra, denoted as $(\Lambda W,\bar{d})$ or simply $\Lambda W$. If $\Lambda Z\otimes\Lambda W$ is a minimal relative Sullivan algebra, then $\Lambda W$ is a minimal Sullivan algebra.

Let $L_0$ denote the space of degree 0 whose suspension $sL_0=(Z^1)^{\sharp}$. This space carries a Lie algebra structure corresponding to the differential on $Z^1$, so it is called the fundamental Lie algebra of $\Lambda Z$. There exists a left $L_0$-action on $H^*(\Lambda W)$ defined as follows. For a cocycle $\Phi\in\Lambda W$,
$$
d(1\otimes\Phi) = \sum z_i\otimes\Theta_i+\Psi,
$$
where $z_i\in Z$, $\Theta_i\in\Lambda W$ and $\Psi\in\Lambda^{\geq 2} Z\otimes\Lambda W$. Define $x\cdot[\Phi]$ as the cohomology class of $\sum_i sx(z_i)\Theta_i$, where $sx\in(Z^1)^{\sharp}$ is considered as a function on $Z$ that vanishes on $Z^{\geq 2}$. It has been shown that this action is indeed a Lie algebra action, and $x\cdot[\Phi]$ is independent of the choice of representative $\Phi$ \cite[Lemma 4.8]{FHT2}.

As stated in Theorem \ref{bij between homotopies} (iii), the bijection $\text{exp}_L\circ s$ translates the Lie algebra $L_0$ into the group $\pi_1(\Lambda Z)$ when $\dim H^1(\Lambda Z)<\infty$. This group is isomorphic to $\pi_1(|\Lambda Z|)$. The above $L_0$-action also induces a group action on $H^*(\Lambda W)$ \cite[Proposition 4.5]{FHT2}, given by
$$
(\text{exp}_L(sx))\cdot[\Phi] = \sum_{i=0}^{\infty} \frac{1}{i!} x^i\cdot \left[\Phi \right].
$$
Here $x^i\cdot\Phi$ denotes the action of $x$ on $\Phi$ $i$ times, i.e. $x\cdot(x\cdot(\ldots(x\cdot\Phi)\ldots))$. By definition if $\Phi\in W(k)$, a representative of $x\cdot[\Phi]$ is spanned by $\Theta_i\in\Lambda W(k-1)$. Consequently, $x^i\cdot[\Phi]=0$ when $i$ is sufficiently large. Therefore, the sum $\sum_{i=0}^{\infty} \frac{1}{i!} x^i\cdot[\Phi]$ is finite, ensuring that the $L_0$-action is well-defined.

Given a minimal Sullivan algebra $\Lambda V$, the subalgebra $\Lambda V^{\leq n}$ can be viewed as a relative Sullivan algebra $\Lambda V^{\leq n-1}\otimes\Lambda V^n$ for $n\geq 2$. In this context, $H^n(\Lambda V^n)=V^n$ can be acted by the fundamental Lie algebra of $\Lambda V^{\leq n-1}$, and by the group $\pi_1(|\Lambda V^{\leq n-1})|$ if $\dim H^1(\Lambda V^{\leq n-1})<\infty$. Then it induces a right $\pi_1(|\Lambda V^{\leq n-1})|$-action on $(V^n)^{\sharp} = \pi_n(|\Lambda V^n|) = H_n(|\Lambda V^n|)$.

On the other hand, we have a fibration $|\Lambda V^n| \to |\Lambda V^{\leq n}| \to |\Lambda V^{\leq n-1}|$ according to Theorem \ref{bij between homotopies}. The holonomy action also induces a right $\pi_1(|\Lambda V^{\leq n-1})|$-action on $\pi_n(|\Lambda V^n|)$ and $H_n(|\Lambda V^n|)$. These actions actually agree.

\begin{thm}\label{holonomy}
Let $\Lambda V$ be a minimal Sullivan algebra with $\dim H^1(\Lambda V)<\infty$. Denote the $\pi_1(|\Lambda V^{\leq n-1}|)$-action on $\Lambda V^n$ of the relative Sullivan algebra $\Lambda V^{\leq n-1}\otimes\Lambda V^n$ by $\cdot$, and denote the holonomy action of the fibration $|\Lambda V^n| \to |\Lambda V^{\leq n}| \to |\Lambda V^{\leq n-1}|$ by $\bullet$. Then for each $\alpha\in\pi_1(|\Lambda V^{\leq n-1})|$, $\beta\in\pi_n(|\Lambda V^n|)$ and $v\in V^n$, the following equations hold.
\begin{align*}
\iota_n(\beta\bullet\alpha) &= \iota_n(\beta)\cdot\alpha, \\
hur(\beta\bullet\alpha) &= hur(\beta)\bullet\alpha, \\
[m_{|\Lambda V^n|}(\alpha\cdot v)] &= \alpha\bullet[m_{|\Lambda V^n|}(v)].
\end{align*}
Moreover, if $m_{|\Lambda V^n|}$ is a quasi-isomorphism, the holonomy action on $H^*(|\Lambda V^n|)$ is nilpotent.
\end{thm}

\begin{rmk}
This theorem is a special case of several results in \cite{FHT2}, including Theorem 4.2, Corollary 4.3, and the discussion in Section 4.6.
\end{rmk}

We end this section by giving the relationship between fibrations and relative Sullivan algebras.

\begin{thm}[Theorem 5.1 and Corollary 5.1 of \cite{FHT2}]\label{fibration and relative Sullivan}
Suppose that $F\to E\to B$ is a fibration satisfying the following conditions.
\begin{itemize}
\item $F$, $E$ and $B$ are path connected.
\item The holonomy action of $\pi_1(B)$ on $H^*(F)$ is nilpotent.
\item Either $H^*(F)$ or $H^*(B)$ is of finite type.
\end{itemize}
Let $\Lambda Z$ be a minimal Sullivan model of $B$, and let $\Lambda Z\otimes\Lambda W$ be a minimal relative Sullivan algebra such that the following diagram commutes.
$$
\xymatrix{
\Lambda Z \ar[r] \ar[d]^{m_B}_{\simeq} & \Lambda Z\otimes\Lambda W \ar[r] \ar[d]^{m_E} & \Lambda W \ar[d]^{m_F} \\
A_{PL}(B) \ar[r] & A_{PL}(E) \ar[r] & A_{PL}(F)
}
$$

(i) If $m_E$ is a quasi-isomorphism, then so is $m_F$.

(ii) If $m_F$ is a quasi-isomorphism, then so is $m_E$.
\end{thm}

\subsection{Sullivan spaces}
Given a path-connected space $X$, there exists a unique minimal Sullivan algebra $\Lambda V$ along with a quasi-isomorphism $\phi:\Lambda V\stackrel{\simeq}{\longrightarrow} A_{PL}(X)$. This minimal Sullivan algebra is unique up to isomorphism, and is called the \textbf{minimal Sullivan model} of $X$. The quasi-isomorphism $\phi$ induces morphisms on homotopy groups $\pi_n(\phi):\pi_n(X) \to \pi_n(\Lambda V)$, where $\pi_1(\phi)$ is defined only when $\dim H^1(X)<\infty$.

If $X$ is simply connected and $H^*(X)$ is of finite type, then the maps
$$
\pi_n(\phi)\otimes\mathbb{Q}: \pi_n(X)\otimes\mathbb{Q} \stackrel{\cong}{\longrightarrow} \pi_n(\Lambda V)
$$
are isomorphisms for $n\geq 2$ (Theorem 10.1 of \cite{Sullivan}, also see Theorem 1.6 of \cite{FHT2}).

When $X$ is not simply connected, $\pi_n(\phi)\otimes\mathbb{Q}$ are not necessarily isomorphic. If $\pi_n(X)\otimes\mathbb{Q}$ can still be reflected by the minimal Sullivan algebra, we call $X$ a Sullivan space. Precisely,

\begin{defn}\label{def of Sullivan space}
A path connected space $X$ is called a \textbf{Sullivan space} if the following conditions are satisfied.
\begin{enumerate}
\item $\dim H^1(X)<\infty$.
\item The universal cover $\widetilde{X}$ of $X$ has finite type cohomology.
\item Given a minimal Sullivan model $\phi:\Lambda V\to A_{PL}(X)$, $\pi_n(\phi)\otimes\mathbb{Q}$ are isomorphisms for $n\geq 2$.
\end{enumerate}
\end{defn}

The following theorem outlines some properties of Sullivan spaces, and can be used to determine whether a given space is a Sullivan space.

\begin{thm}[Theorem 7.2 of \cite{FHT2}]\label{thm 7.2 FHT}
Suppose $X$ is a CW complex. Let $F\to E\to B$ be its classifying space fibration, where $F=\widetilde{X}$ is the universal cover of $X$, $E$ is a space homotopy equivalent to $X$, and $B=K(\pi_1(X),1)$ is the classifying space of $\pi_1(X)$. The minimal model $m_E:\Lambda V\stackrel{\simeq}{\longrightarrow} A_{PL}(E)$ induces the following commutative diagram.
$$
\xymatrix{
\Lambda V^1 \ar[r] \ar[d]^{m_B} & \Lambda V \ar[r] \ar[d]^{m_E}_{\simeq} & \Lambda V^{\geq 2} \ar[d]^{m_F} \\
A_{PL}(B) \ar[r] & A_{PL}(E) \ar[r] & A_{PL}(F)
}
$$

If $X$ is a Sullivan space, the following statements all holds.
\begin{enumerate}[(i)]
\item $m_B$ is a quasi-isomorphism and $H^*(B)$ has finite type.
\item $m_F$ is a quasi-isomorphism.
\item $\dim H^1(X)<\infty$ and $H^*(\widetilde{X})=H^*(F)$ are of finite type.
\item $H^*(X)$ is of finite type.
\item $\pi_1(X)$ acts nilpotently via covering transformation on each $H^k(\widetilde{X})$, i.e. the holonomy action of $\pi_1(B)$ on each $H^k(F)$ is nilpotent.
\end{enumerate}

Conversely, $X$ is a Sullivan space if one of the following collections of statements holds.

1. (i)(ii)(iii) \qquad 2. (i)(iii)(v) \qquad 3. (i)(iv)(v).
\end{thm}

As the following theorem states, for a Sullivan algebras $\Lambda V$ satisfying certain conditions, its realization $|\Lambda V|$ being a Sullivan space is equivalent to $m_{|\Lambda V|}: \Lambda V \to A_{PL}(|\Lambda V|)$ being a quasi-isomorphism. In \cite{FHT2}, F\'elix, Halperin, and Thomas posed the question of whether this equivalence holds for all minimal Sullivan algebras.

\begin{thm}[Theorem 7.8 of \cite{FHT2}]\label{thm 7.8 FHT}
For any minimal Sullivan algebra $\Lambda V$ the following conditions are equivalent.
\begin{enumerate}[(i)]
\item $H^1(\Lambda V)$ and each $V^i$ for $i\geq 2$ are finite dimensional, and $m_{|\Lambda V|}: \Lambda V \to A_{PL}(|\Lambda V|)$ is a quasi-isomorphism.
\item $|\Lambda V|$ is a Sullivan space.
\end{enumerate}
\end{thm}

In particular, if $\Lambda V$ is of finite type, then $|\Lambda V|$ is a Sullivan space.

\begin{prop}[Theorem 8.1 of \cite{Sullivan}, also see Theorem 5.4 of \cite{FHT2}]\label{finite type quasi-iso}
Let $\Lambda V$ be a minimal Sullivan algebra of finite type. Then $|\Lambda V|$ is a Sullivan space and $m_{|\Lambda V|}: \Lambda V \to A_{PL}(|\Lambda V|)$ is a quasi-isomorphism.
\end{prop}

\section{Cohomologies of a minimal Sullivan algebra and its realization}
In this section, we prove Theorem \ref{intro cohomology of minimal Sullivan and realization}, i.e., we find an equivalent condition for statement (i) of Theorem \ref{thm 7.8 FHT} without the restrict on dimensions.

\begin{lem}[Exercise 3.5.2 of \cite{weibel}]\label{ex 3.5.2 of weibel}
Let
$$
\ldots \to V_i \to V_{i-1} \to \ldots \to V_1 \to V_0
$$
be a tower of finite dimensional vector spaces over a field. Then the derived functor on the inverse limit satisfies
$$
\varprojlim_i\textstyle{^1} \, V_i=0.
$$
\end{lem}

\begin{lem}[Theorem 3.5.8 of \cite{weibel}]\label{thm 3.5.8 of weibel}
Let
$$
\ldots \to C_i \to C_{i-1} \to \ldots \to C_1 \to C_0
$$
be a tower of chain complexes satisfying the Mittag-Leffler condition (e.g., if all maps $C_i \to C_{i-1}$ are surjective). Let $\displaystyle C=\varprojlim_i C_i$. Then there exists an exact sequence for each $n$:
$$
0 \to \varprojlim_i\textstyle{^1} \displaystyle\, H_{n+1}(C_i) \to H_n(C) \to \varprojlim_i H_n(C_i) \to 0.
$$
\end{lem}

\begin{thm}\label{homology finite type sufficient}
Let $\Lambda V$ be a minimal Sullivan algebra with finite type cohomology. Then $m_{|\Lambda V|}: \Lambda V \to A_{PL}(|\Lambda V|)$ is a quasi-isomorphism.
\end{thm}
\begin{proof}
Equip $V$ with the filtration $V(r)$ as follows. Set $V(-1)=0$ and define
$$
V(r) = \{ v\in V \mid dv\in \Lambda V(r-1) \} \text{ for } r\geq 0.
$$
Since $\Lambda V$ is minimal, all coboundaries have wordlength at least 2. So the natural map $V(0)\to H^*(\Lambda V(0))$, which sends cocycles to its cohomology classes, is injective. It follows that $V(0)$ and $\Lambda V(0)$ are of finite type. Using induction on $r$, we can prove that all $\Lambda V(r)$ are of finite type.

By Proposition \ref{finite type quasi-iso}, $m_{|\Lambda V(r)|}: \Lambda V(r) \stackrel{\simeq}{\longrightarrow} A_{PL}(|\Lambda V(r)|)$ are quasi-isomorphisms for all $r$, and all $H^*(|\Lambda V(r)|)$ have finite type. Consequently, the inclusions $\Lambda V(0) \to \Lambda V(1) \to \ldots$ induce a on each degree $n$,
$$
\ldots \to H_n(|\Lambda V(r)|) \to H_n(|\Lambda V(r-1)|) \to \ldots \to H_n(|\Lambda V(1)|) \to H_n(|\Lambda V(0)|)
$$
by taking the dual on cohomologies. According to \eqref{identify cohomologies} and Lemma \ref{ex 3.5.2 of weibel}, we have
$$
\varprojlim_r\textstyle{^1} \displaystyle\, H_n(C_*(\langle\Lambda V(r)\rangle)) = \varprojlim_r\textstyle{^1} \, H_n(|\Lambda V(r)|)=0.
$$

We claim that the map $C_*(\langle\Lambda V(r)\rangle) \to C_*(\langle\Lambda V(r-1)\rangle)$ is surjective. Recall that $\langle\Lambda V(r)\rangle_n$ consists of morphisms from $\Lambda V(r)$ to $(A_{PL})_n$. The map $\langle\Lambda V(r)\rangle_n \to \langle\Lambda V(r-1)\rangle_n$ sends these morphisms to their restrictions to $\Lambda V(r-1)$. Choose a basis $\{[v_j]\}$ of $V(r)/V(r-1)$ and take representative $v_j\in V(r)$ for each $v_j$. Given a morphism $f:\Lambda V(r-1)\to(A_{PL})_n$, $f(dv_j)$ is a cocycle in $(A_{PL})_n$ of positive degree. It must be a coboundary because $H^{\geq 1}((A_{PL})_n)=0$. Thus, we can define $f(v_j)$ to satisfy $df(v_j)=f(dv_j)$. This allows us to extend $f$ to $\Lambda V(r)$, proving the claim.

Since $V=\displaystyle\varinjlim_r V(r)$, we have $C_*(\langle\Lambda V\rangle) = \displaystyle\varprojlim_r C_*(\langle\Lambda V(r)\rangle)$. Applying Lemma \ref{thm 3.5.8 of weibel}, we obtain $H_n(C_*(\langle\Lambda V\rangle)) = \displaystyle\varprojlim_r H_n(C_*(\langle\Lambda V(r)\rangle))$. Then taking the dual, we get $H^n(|\Lambda V|) = \displaystyle\varinjlim_r H^n(|\Lambda V(r)|)$.

Let $x\in\Lambda V$ be a cocycle such that $m_{|\Lambda V|}(x)$ is a coboundary in $A_{PL}(|\Lambda V|)$. Then $x$ must be contained in some $\Lambda V(r)$. As discussed, $H^n(|\Lambda V|)$ is the direct limit of $H^n(|\Lambda V(r)|)$. So there exists some $r'\geq r$ such that $m_{|\Lambda V(r')|}(x)$ is a coboundary in $A_{PL}(|\Lambda V(r')|)$. Since $m_{|\Lambda V(r')|}$ is a quasi-isomorphism, $x$ has to be a coboundary in $\Lambda V(r')$, and hence also a coboundary in $\Lambda V$. This proves that $m_{|\Lambda V|}$ is injective on cohomologies.

On the other hand, $H^n(|\Lambda V|) = \displaystyle\varinjlim_r H^n(|\Lambda V(r)|)$ also implies that every $[y]\in H^n(|\Lambda V|)$ must appear at some $H^n(|\Lambda V(r)|)$. That is, there exists a cocycle $y_r\in A_{PL}(|\Lambda V(r)|)$ such that the induced morphism $H^n(|\Lambda V(r)|) \to H^n(|\Lambda V|)$ sends the cohomology class $[y_r]$ to $[y]$. Consider the following commutative diagram.
$$ 
\xymatrix{
H^n(\Lambda V(r)) \ar[d]^{H(m_{|\Lambda V(r)|})}_{\cong} \ar[r] & H^n(\Lambda V) \ar[d]^{H(m_{|\Lambda V|})} \\
H^n(|\Lambda V(r)|) \ar[r] & H^n(|\Lambda V|)
}
$$
There exists some $[x_r]$ that maps to $[y_r]$ via the isomorphism $H(m_{|\Lambda V(r)|}): H^*(\Lambda V(r)) \to H^*(|\Lambda V(r)|)$. Let $[x]\in H^*(\Lambda V)$ denote the image of $[x_r]$ in the direct limit. Then $H(m_{|\Lambda V|})$ maps $[x]$ to $[y]$. Thus, $m_{|\Lambda V|}$ is surjective on cohomologies.
\end{proof}

The following example provides a counterexample to the first part of Question \ref{FHT question}.

\begin{ex}\label{eg S1 vee S2}
Let $X=S^1\vee S^2$ and let $\Lambda V$ be its minimal Sullivan model. Then $H^*(\Lambda V)=H^*(X)$ is finite dimensional. By Theorem \ref{homology finite type sufficient}, $m_{|\Lambda V|}$ is a quasi-isomorphism, and the induced map $|\Lambda V|\to X$ is isomorphic on cohomologies.

However, $|\Lambda V|$ is not a Sullivan space because the cohomology of its universal cover, which is homotopy equivalent to $|\Lambda V^{\geq 2}|$, does not have finite type. Observe that $V^1=\langle x\rangle$, and $V^2=\langle y_0,y_1,\ldots,y_i,\ldots\rangle$, where $dx=dy_0=0$, $dy_i=xy_{i-1}$ for $i\geq 1$. Consequently, $\pi_2(|\Lambda V|)\otimes\mathbb{Q}$ is infinite dimensional, and so is $H_2(|\Lambda V^{\geq 2}|)$. This contradicts statement (iii) of Theorem \ref{thm 7.2 FHT}. Alternatively, the differential on $\Lambda V$ induces a non-nilpotent action on $H^2(|\Lambda V^{\geq 2}|)$ by $\pi_1(|\Lambda V|)$, which contradicts statement (v) of that theorem.
\end{ex}

\begin{thm}\label{homology finite type}
Let $\Lambda V$ be a minimal Sullivan algebra. If there exists a quasi-isomorphism $\Lambda V \stackrel{\simeq}{\longrightarrow} A_{PL}(|\Lambda V|)$, then $H^*(\Lambda V)=H^*(|\Lambda V|)$ is of finite type.
\end{thm}
\begin{proof}
Suppose there exists a quasi-isomorphism $\Lambda V \stackrel{\simeq}{\longrightarrow} A_{PL}(|\Lambda V|)$ and $H^*(\Lambda V)$ is not of finite type. Then let $n$ be the smallest integer for which $\dim H^n(\Lambda V) = \dim H^n(|\Lambda V|) = \infty$. Since $H^n(|\Lambda V|) = [H_n(|\Lambda V|)]^{\sharp}$, the dimension of $H^n(|\Lambda V|)$ must be uncountable. We denote this dimension by the cardinal number $\aleph$.

\noindent\textbf{Claim.} $n$ is also the smallest number such that $\dim V^n$ is uncountable. Additionally, the image of $\xi_n: H^n(\Lambda V) \to V^n$ has dimension $\aleph$.

Assume that $i<n$ is the smallest integer such that $\dim V^i$ is uncountable. Then $\dim (\Lambda V^{\leq i-1})^{i+1}$ is at most countable. Since $\Lambda V$ is minimal, the kernel of $d:V^i\to(\Lambda V^{\leq i-1})^{i+1}$ has uncountable dimension. Then so does $H^i(\Lambda V)$, which contradicts the hypothesis. Therefore, $V^{\leq n-1}$ must be of at most countable dimension, and so must be the cycles in $(\Lambda V^{\leq n-1})^{n+1}$. Given that $\dim H^n(\Lambda V)=\aleph$, which is uncountable, there must be an $\aleph$-dimensional subspace of $V^n$ containing all cycles. In particular, this subspace is $\im\xi_n$, which proves the claim.

As the restriction of $\xi_n^{\sharp}$ to $(\im\xi_n)^{\sharp}$ is injective, the dimension of $\im\xi_n^{\sharp}$ is at least $\dim (\im\xi_n)^{\sharp}$, which is strictly greater than $\aleph$.

When $n\geq 2$, by Theorem \ref{bij between homotopies}, we have $\im(\xi_n^{\sharp}\circ\iota_n)=\im\xi_n^{\sharp}$. So the dimension of this image is strictly greater than $\aleph$. On the other hand, $H^n(|\Lambda V|)$ is the infinite dimensional dual space of $H_n(|\Lambda V|)$. Therefore, $\dim H_n(|\Lambda V|)$ is strict smaller than $\aleph$, and so is $\dim \im(H^n\left(m_{|\Lambda V|}\right)^{\sharp}\circ hur)$. But $\xi_n^{\sharp}\circ\iota_n = H^n\left(m_{|\Lambda V|}\right)^{\sharp}\circ hur$, leading a contradiction.

When $n=1$, by Lemma \ref{commutative for deg 1}, the dimension of $\im (H^1 \left(m_{|\Lambda V|}\right)^{\sharp}\circ hur\circ\tau_1) = \im\xi_1^{\sharp}$ is strictly greater than $\aleph$. However, $\im (hur\circ\tau_1) \subset H_1(|\Lambda V|)$. Thus,
$$
\dim \im (H^1 \left(m_{|\Lambda V|}\right)^{\sharp}\circ hur\circ\tau_1) \leq \dim \im (hur\circ\tau_1) \leq \dim H_1(|\Lambda V|) < \aleph,
$$
which is a contradiction.
\end{proof}

Combining Theorem \ref{homology finite type sufficient} with Theorem \ref{homology finite type}, we have established Theorem \ref{intro cohomology of minimal Sullivan and realization}.

\begin{thm}
Let $\Lambda V$ be a minimal Sullivan algebra. The following statements are equivalent.
\begin{enumerate}[(i)]
\item There exists a quasi-isomorphism $\Lambda V \stackrel{\simeq}{\longrightarrow} A_{PL}(|\Lambda V|)$.
\item $m_{|\Lambda V|}:\Lambda V \to A_{PL}(|\Lambda V|)$ is a quasi-isomorphism.
\item $H^*(\Lambda V)$ has finite type.
\end{enumerate}
\end{thm}

Realization and taking minimal Sullivan model also translate corresponding morphisms into each other, which allows the above theorem to be expressed in the language of the category theory. A CDGA morphism $\phi:\Lambda V\to\Lambda W$ induces a continuous map $|\phi|:|\Lambda W|\to|\Lambda V|$. Moreover, the homotopy class of $|\phi|$ depends only on the homotopy class of $\phi$ \cite[Proposition 17.13]{FHT}. Conversely, given a continuous map $f:X\to Y$ between path connected spaces, and minimal Sullivan models $\Lambda V\stackrel{\simeq}{\longrightarrow} A_{PL}(X)$ and $\Lambda W\stackrel{\simeq}{\longrightarrow} A_{PL}(Y)$, there exists a CDGA morphism $\psi:\Lambda W\to\Lambda V$ unique up to homotopy making the following diagram homotopy commutative.

\begin{equation}\label{Sullivan representative for morphism}
\xymatrix@C=60pt{
\Lambda W \ar[r]^{\psi} \ar[d]^{\simeq} & \Lambda V \ar[d]^{\simeq} \\
A_{PL}(Y) \ar[r]^{A_{PL}(f)} & A_{PL}(X)
}
\end{equation}

Let $\mathcal{C}$ be the category whose objects are isomorphism classes of minimal Sullivan algebras with finite type cohomology, and whose morphisms are homotopy classes of CDGA morphisms. Let $\mathcal{D}$ be the category whose objects are path connected spaces with finite type cohomology, and whose morphisms are homotopy classes of continuous maps. The above discussion shows that realization $\mathcal{F}:\mathcal{C}\to \mathcal{D}$ and taking minimal Sullivan models $\mathcal{G}:\mathcal{D}\to \mathcal{C}$ are well-defined contravariant functors.

Theorem \ref{intro cohomology of minimal Sullivan and realization} then implies that $\mathcal{G}$ is a left inverse of $\mathcal{F}$. Hence, $\mathcal{F}$ is a faithful functor. That $\mathcal{F}$ is also full follows from the discussion below.

Given a CW complex $X$, the adjoint of the identify map on $Sing X$ is a homotopy equivalence $\lambda_X: |Sing X| \to X$ (\cite[Theorem 16.6(ii)]{may}, see also \cite[Proposition 1.5(iii)]{FHT2}). It has an inverse homotopy equivalence $\nu_X:X\to|Sing X|$ unique up to homotopy. Let $m_X:\Lambda V\stackrel{\simeq}{\longrightarrow} A_{PL}(X)$ be a minimal Sullivan model. The realization of its adjoint is $|m_X|:|Sing X|\to|\Lambda V|$. $\|m_X\|=|m_X|\circ\nu_X:X\to|\Lambda V|$ is a morphism in $\mathcal{D}$, independent of the choice of $\nu_X$.

\begin{prop}[Theorem 17.15 of \cite{FHT}]\label{thm 17.15 of FHT}
Let $f:X\to Y$ be a continuous map between path connected spaces, and let $m_X:\Lambda V\stackrel{\simeq}{\longrightarrow} A_{PL}(X)$, $m_Y:\Lambda W\stackrel{\simeq}{\longrightarrow} A_{PL}(Y)$ be minimal Sullivan models. Consequently, $f$ has a Sullivan representative $\psi:\Lambda W\to\Lambda V$ making \eqref{Sullivan representative for morphism} homotopy commutative. Then the following diagram is also homotopy commutative.
$$
\xymatrix@C=60pt{
X \ar[r]^{f} \ar[d]^{\|m_X\|} & Y \ar[d]^{\|m_Y\|} \\
|\Lambda V| \ar[r]^{|\psi|} & |\Lambda W|
}
$$
\end{prop}

\begin{rmk}
Although Theorem 17.15 of \cite{FHT} assumes that $X$ and $Y$ are simply connected and have finite type homologies, the proof extends to the general case without modification.
\end{rmk}

Now let $X=|\Lambda V|$ and $Y=|\Lambda W|$ with finite type cohomologies. By Theorem \ref{homology finite type sufficient} $m_{|\Lambda V|}: \Lambda V \to A_{PL}(|\Lambda V|)$ and $m_{|\Lambda W|}: \Lambda W \to A_{PL}(|\Lambda W|)$ are quasi-isomorphisms, hence minimal Sullivan models. Recall that $m_{|\Lambda V|}$ is a lift of $m_{\langle\Lambda V\rangle}:\Lambda V \to A_{PL}(\langle\Lambda V\rangle)$, which is the adjoint of the identity on $\langle\Lambda V\rangle$. Thus $\lambda_{\langle\Lambda V\rangle}: \langle\Lambda V\rangle \to Sing(|\Lambda V|)$, the adjoint of the identity on $|\Lambda V|$, serves as a left inverse to $\langle m_{|\Lambda V|} \rangle: Sing(|\Lambda V|) \to \langle\Lambda V\rangle $, which is the adjoint of $m_{|\Lambda V|}$. Consequently, the composition of realizations
$$
\xymatrix@C=40pt{
|Sing(|\Lambda V|)| \ar[r]^(.6){|m_{|\Lambda V|}|} & |\Lambda V| \ar[r]^(.4){|\lambda_{\langle\Lambda V\rangle}|} & |Sing(|\Lambda V|)|
}
$$
is the identity map on $|Sing(|\Lambda V|)|$.

On the other hand, by definition $|\lambda_{\langle\Lambda V\rangle}|$ is a right inverse to the homotopy equivalence $\lambda_{|\Lambda V|}:|Sing(|\Lambda V|)| \to |\Lambda V|$. Thus, we can take $\nu_{|\Lambda V|}=|\lambda_{\langle\Lambda V\rangle}|$, making $|m_{|\Lambda V|}|$ its homotopy inverse. Then $\|m_{|\Lambda V|}\| = |m_{|\Lambda V|}|\circ\nu_{|\Lambda V|} \sim id_{|\Lambda V|}$. For the same reason, $\|m_{|\Lambda W|}\| \sim id_{|\Lambda W|}$.

Therefore, applying Proposition \ref{thm 17.15 of FHT}, we get
$$
f \sim \|m_{|\Lambda W|}\|\circ f \sim |\psi|\circ\|m_{|\Lambda V|}\| \sim |\psi|.
$$
The homotopy class of $f$ is the same as that of $|\psi|$, which is the image of the homotopy class of $\psi$ under $\mathcal{F}$. This proves that $\mathcal{F}$ is full. In conclusion, we have the following statement.

\begin{thm}
$\mathcal{F}:\mathcal{C}\to\mathcal{D}$ is a full and faithful functor, with $\mathcal{G}$ as its left inverse.
\end{thm}

\section{Homotopy groups of a topological space and its minimal Sullivan model}

We can also investigate the relationship between the homotopy groups of a topological space $X$ and those of its minimal Sullivan model $\Lambda V$. Suppose $\dim H^1(X)<\infty$ and the cohomology $H^*(\widetilde{X})$ of the universal cover of $X$ is of finite type. In this case, spaces that induce isomorphisms $\pi_n(X)\otimes\mathbb{Q} \stackrel{\cong}{\longrightarrow} \pi_n(\Lambda V)$ for all $n\geq 2$ are known as Sullivan spaces (Definition \ref{def of Sullivan space}).  Without these cohomological conditions, there exist other spaces which share the same higher rational homotopy groups as their minimal Sullivan models, although they do not necessarily satisfy the five properties of Theorem \ref{thm 7.2 FHT}.

According to Theorem \ref{homology finite type sufficient}, if $X$ is weak homotopy equivalent to the realization of its minimal model $|\Lambda V|$, then $\Lambda V\to A_{PL}(X)$ induces isomorphisms on homotopy groups. For instance, $\Lambda V$ is the minimal model of $S^1\vee S^2$ in Example \ref{eg S1 vee S2}. In this case, $H^*(\tilde{X})=H^*(|\Lambda V^{\geq 2}|)$ does not have finite type, so Statement (iii) of Theorem \ref{thm 7.2 FHT} is not satisfied. Consequently, $m_{|\Lambda V^{\geq 2}|}:\Lambda V^{\geq 2} \to A_{PL}(|\Lambda V^{\geq 2}|)$ can not be a quasi-isomorphism (Theorem 1.5 of \cite{FHT2}), which means (ii) is not satisfied. Moreover, the $\pi_1(X)$-action on $H^*(\tilde{X})$ is not nilpotent, thus violating (v).

The following example demonstrates that $\dim H^1(X)$ can be infinity while the higher homotopy group of $X$ and its minimal model remain the same. Hence, (i) and (iv) of Theorem \ref{thm 7.2 FHT} are not satisfied.

\begin{ex}
Let $X=\bigvee_{i=1}^{\infty} S^1_i$ be an infinite wedge of circles. Then
$$
\dim H^n(X)=
\begin{cases}
1, & \text{if } n=0, \\
\infty, & \text{if } n=1, \\
0, & \text{if } n\geq 2. \\
\end{cases}
$$
Since $H^*(X)=H^1(X)\oplus\mathbb{Q}$, by Lemma 7.7 of \cite{FHT2}, the minimal Sullivan model of $X$ is $\Lambda V^1$ where $V^1$ is concentrated in degree 1. Consequently, both $\pi_n(X)$ and $\pi_n(\Lambda V)$ are trivial for $n\geq 2$.
\end{ex}

\begin{rmk}
Statement (i) and (iii) of Theorem \ref{thm 7.2 FHT} each contain two parts. When we require only that $\pi_n(X)\otimes\mathbb{Q} \cong \pi_n(\Lambda V)$, the above counterexamples fail both parts of (iii). However, whether $\Lambda V^1 \to A_{PL}(B\pi_1(X))$ in statement (i) is always a quasi-isomorphism remains unknown.
\end{rmk}

Given that the fundamental Lie algebra of a minimal Sullivan algebra $\Lambda V$ is related to $\pi_1(|\Lambda V|)$, we may investigate the conditions under which a similar relationship holds between $\pi_1(X)$ of a space $X$ and the fundamental Lie algebra of its minimal Sullivan model. By utilizing Malcev $\mathbb{Q}$-completions, we can establish a criterion that applies to spaces, including Sullivan spaces, while excluding those $X$ for which $\dim H^1(X)=\infty$.

A group $\Gamma$ is called Malcev $\mathbb{Q}$-complete if each $\Gamma^n/\Gamma^{n+1}$ is a $\mathbb{Q}$-vector space and $\Gamma = \varprojlim \Gamma^n/\Gamma^{n+1}$. The Malcev $\mathbb{Q}$-completion of a group $G$ is a homomorphism $\varphi:G\to\Gamma$, where $\Gamma$ is Malcev $\mathbb{Q}$-complete and $\varphi$ induces isomorphisms
$$
(G^n/G^{n+1}) \otimes_{\mathbb{Z}} \mathbb{Q} \stackrel{\cong}{\longrightarrow} \Gamma^n/\Gamma^{n+1}.
$$

By applying Corollary 2.4 and Corollary 7.4 of \cite{FHT2}, we obtain the following proposition concerning $\pi_1(X)$ when $\dim H^1(X)<\infty$.

\begin{prop}
Let $\Lambda V$ be a minimal Sullivan model of a topological space $X$. The induced morphism $\pi_1(X)\to\pi_1(\Lambda V)$ is a Malcev $\mathbb{Q}$-completion. Moreover, Let $L_0$ denote the fundamental Lie algebra of $\Lambda V$. Then there are isomorphisms of abelian groups
$$
(\pi_1(X)^n/\pi_1(X)^{n+1})\otimes_{\mathbb{Z}} \mathbb{Q} \cong L_0^n/L_0^{n+1}
$$
for all $n$.
\end{prop}

On the other hand, the isomorphisms between $(\pi_1(X)^n/\pi_1(X)^{n+1})\otimes_{\mathbb{Z}} \mathbb{Q}$ and $L_0^n/L_0^{n+1}$ do not hold if $\dim H^1(X)=\infty$.

\begin{prop}
Let $\Lambda V$ be a minimal Sullivan model of a topological space $X$, and let $L_0$ denote the fundamental Lie algebra of $\Lambda V$. If $\dim H^1(X)=\infty$, then $H_1(X) = (\pi_1(X)/[\pi_1(X),\pi_1(X)])\otimes_{\mathbb{Z}} \mathbb{Q}$ and $L_0/[L_0,L_0]$ are not isomorphic.
\end{prop}
\begin{proof}
Equip $V$ with the filtration $V(r)$ by setting $V(-1)=0$ and defining
$$
V(r) = \{ v\in V \mid dv\in\Lambda V(r-1) \}.
$$
Then $[L_0,L_0]$ vanishes on $V(0)=H^1(\Lambda V)$ \cite[Lemma 2.1]{FHT2}. As $sL_0=(V^1)^{\sharp}$, this induces a surjective morphism $L_0/[L_0,L_0] \to [H^1(\Lambda V)]^{\sharp}$. Therefore, the cardinal number of $\dim L_0/[L_0,L_0]$ is greater than or equal to the cardinal number of $\dim [H^1(\Lambda V)]^{\sharp}=\dim [H^1(X)]^{\sharp}$, which is strict larger than the cardinal number of $\dim H_1(X)$ due to the hypothesis $\dim H^1(X)=\infty$. Consequently, there can be no isomorphism between $H_1(X)$ and $L_0/[L_0,L_0]$.
\end{proof}

These propositions characterize spaces which induce isomorphisms on the higher rational homotopy groups and a Malcev $\mathbb{Q}$-completion on the fundamental group to their minimal Sullivan models. Such spaces represent a reasonable generalization of Sullivan spaces. For instance, Example \ref{eg S1 vee S2} meets these criteria but is not a Sullivan space, indicating that this generalization is non-trivial. This also demonstrates that these spaces do not necessarily fulfill Statements (ii), (iii), and (v) of Theorem \ref{thm 7.2 FHT}.

It is current unknown whether both parts of Statement (i) hold for such spaces. However, using the approach outlined in Theorem \ref{homology finite type}, we can prove that these spaces satisfy Statement (iv). This result is stated in Proposition \ref{intro finite type of H(X)} in the introduction.

\begin{prop}\label{finite type of H(X)}
Let $f:\Lambda V \stackrel{\simeq}{\longrightarrow} A_{PL}(X)$ be a minimal Sullivan model of a topological space $X$, with $\dim H^1(X) = \dim H^1(\Lambda V) < \infty$. If for all $n\geq 2$, the induced map $\pi_n(f):\pi_n(X)\otimes\mathbb{Q} \to \pi_n(\Lambda V)$ is an isomorphism, then $H^*(X)$ is of finite type.
\end{prop}
\begin{proof}
Assume that $H^*(X)$ is not of finite type. Then there exists a smallest $n$ such that $\dim H^n(\Lambda V) = \dim H^n(X)=\infty$. This dimension must be uncountable since $H^n(X)=[H_n(X)]^{\sharp}$. Denote this dimension by $\aleph$. As shown in the claim of Theorem \ref{homology finite type}, $n$ is also the smallest number such that $\dim V^n$ is uncountable, and the image of $\xi_n: H^n(\Lambda V) \to V^n$ is $\aleph$-dimensional. Then the restriction of $\xi_n^{\sharp}$ to $(\im\xi_n)^{\sharp}$ is injective, and the cardinal number of $\dim\im\xi_n^{\sharp}$ is strict greater than $\dim \im\xi_n=\aleph$.

The minimal model $f$ induces a map $|f|:X \to |\Lambda V|$, which satisfies $\iota_n\circ\pi_n(|f|) = \pi_n(f)$ for $n\geq 2$ \cite[Diagram (1.21)]{FHT2}. Consequently, $\pi_n(|f|): \pi_n(X)
\to \pi_n(|\Lambda V|)$ is an isomorphism. By the naturality of $hur$ and \eqref{relation with hur}, we obtain the following commutative diagram.
$$
\xymatrix@C=60pt{
\pi_n(X) \ar[r]^{\pi_n(|f|)}_{\cong} \ar[d]^{hur} & \pi_n(|\Lambda V|) \ar[r]^{\iota_n}_{\cong} \ar[d]^{hur} & (V^n)^{\sharp} \ar[d]^{\xi_n^{\sharp}} \\
H_n(X) \ar[r]^{H_n(|f|)} & H_n(|\Lambda V|) \ar[r]^{H^n \left(m_{|\Lambda V|}\right)^{\sharp}} & [H^n(\Lambda V)]^{\sharp}.
}
$$

Thus, the image of $\xi_n^{\sharp}\circ\iota_n\circ\pi_n(|f|)$ is $\im\xi_n^{\sharp}$, and the cardinality of its dimension is strictly larger than $\aleph$. But on the other hand,
\begin{align*}
\dim \im \xi_n^{\sharp}\circ\iota_n\circ\pi_n(|f|) &= \dim \im H^n \left(m_{|\Lambda V|}\right)^{\sharp} \circ H_n(|\Lambda V|) \circ hur \\
& \leq \dim \im H^n \left(m_{|\Lambda V|}\right)^{\sharp} \circ H_n(|\Lambda V|) \\
& \leq \dim H_n(X) \\
& < \aleph,
\end{align*}
which leads to a contradiction.
\end{proof}

The following example demonstrates that the converse of Proposition \ref{finite type of H(X)} does not hold. That $H^*(X)$ has finite type does not necessarily imply that $\pi_n(f):\pi_n(X)\otimes\mathbb{Q} \to \pi_n(\Lambda V)$ is an isomorphism for $n\geq 2$. Thus, although $X$ and $|\Lambda V|$ have the same rational cohomology, their rational higher homotopy groups can be different, even if $|\Lambda V|$ is a Sullivan space.

\begin{ex}
Let $X=\mathbb{R}P^2$. Then $H^*(X)=\mathbb{Q}$. So the minimal Sullivan model of $X$ is trivially $\mathbb{Q}$, and its realization is a point, which is a Sullivan space. However, $\pi_2(X)\otimes\mathbb{Q}=\mathbb{Q}\neq 0=\pi_2(\mathbb{Q})$. Thus, $X$ is not a Sullivan space.
\end{ex}

\bibliographystyle{plain}
\bibliography{refs}

\vskip 1 cm
\noindent
{Beijing Institute of Mathematical Sciences and Applications, Huairou District, Beijing, China 101408}\\
{\it Email address:}~{\tt jiaweizhou@bimsa.cn}
\end{document}